\documentclass[12pt]{amsart}
\usepackage{amsmath,amssymb,amsbsy,amsfonts,amsthm,latexsym,mathabx,
            amsopn,amstext,amsxtra,euscript,amscd,stmaryrd,mathrsfs,
            cite,array,mathtools,enumerate}
\usepackage{url}
\usepackage[colorlinks,linkcolor=blue,anchorcolor=blue,citecolor=blue]{hyperref}

\usepackage{color}
\usepackage{float}

\hypersetup{breaklinks=true}
\AtBeginDocument{%
   \def\MR#1{}
}

\usepackage{etoolbox}

\makeatletter
\patchcmd{\@maketitle}
{\if@titlepage \newpage \else}
{\if@titlepage
 \vspace{\baselineskip}
 \else}
{}{}
\makeatother

\begin{document}

\newtheorem{theorem}{Theorem}
\newtheorem{lemma}[theorem]{Lemma}
\newtheorem{claim}[theorem]{Claim}
\newtheorem{cor}[theorem]{Corollary}
\newtheorem{proposition}[theorem]{Proposition}
\newtheorem{definition}{Definition}
\newtheorem{question}[theorem]{Question}
\newtheorem{remark}[theorem]{Remark}
\newcommand{\hh}{{{\mathrm h}}}

\numberwithin{equation}{section}
\numberwithin{theorem}{section}
\numberwithin{table}{section}

\def\sssum{\mathop{\sum\!\sum\!\sum}}
\def\ssum{\mathop{\sum\ldots \sum}}
\def\dsum{\mathop{\sum \sum}}
\def\iint{\mathop{\int\ldots \int}}

\def\squareforqed{\hbox{\rlap{$\sqcap$}$\sqcup$}}
\def\qed{\ifmmode\squareforqed\else{\unskip\nobreak\hfil
\penalty50\hskip1em\null\nobreak\hfil\squareforqed
\parfillskip=0pt\finalhyphendemerits=0\endgraf}\fi}%%

%  use the AMS-Euler Fraktur fonts
%%%%%%%%%%%%%%%%%%%%%%%%%%%%%%%%%%
\newfont{\teneufm}{eufm10}
\newfont{\seveneufm}{eufm7}
\newfont{\fiveeufm}{eufm5}
%%%%%%%%%%%%%%%%%%%%%%%%%%%%%%%%%
%
%  allow automatic size selection in math mode
%
%%%%%%%%%%%%%%%%%%%%%%%%%%%%%%%%%
\newfam\eufmfam
     \textfont\eufmfam=\teneufm
\scriptfont\eufmfam=\seveneufm
     \scriptscriptfont\eufmfam=\fiveeufm
%%%%%%%%%%%%%%%%%%%%%%%%%%%%%%%%%
%
%  \frak works on a single symbol at a time...
%
\def\frak#1{{\fam\eufmfam\relax#1}}

\newcommand{\bflambda}{{\boldsymbol{\lambda}}}
\newcommand{\bfmu}{{\boldsymbol{\mu}}}
\newcommand{\bfxi}{{\boldsymbol{\xi}}}
\newcommand{\bfrho}{{\boldsymbol{\rho}}}

\def\fK{\mathfrak K}
\def\fT{\mathfrak{T}}

\def\fA{{\mathfrak A}}
\def\fB{{\mathfrak B}}
\def\fC{{\mathfrak C}}

\def\E{\mathsf {E}}

\def \balpha{\bm{\alpha}}
\def \bbeta{\bm{\beta}}
\def \bgamma{\bm{\gamma}}
\def \blambda{\bm{\lambda}}
\def \bchi{\bm{\chi}}
\def \bphi{\bm{\varphi}}
\def \bpsi{\bm{\psi}}

\def\eqref#1{(\ref{#1})}

\def\vec#1{\mathbf{#1}}

%\def\squareforqed{\hbox{\rlap{$\sqcap$}$\sqcup$}}
%\def\qed{\ifmmode\squareforqed\else{\unskip\nobreak\hfil
%\penalty50\hskip1em\null\nobreak\hfil\squareforqed
%\parfillskip=0pt\finalhyphendemerits=0\endgraf}\fi}

%%%%%%%%%%%%%%%%%%%%%%%%%
% Alphabet calligraphie %
%%%%%%%%%%%%%%%%%%%%%%%%%
\def\cA{{\mathcal A}}
\def\cB{{\mathcal B}}
\def\cC{{\mathcal C}}
\def\cD{{\mathcal D}}
\def\cE{{\mathcal E}}
\def\cF{{\mathcal F}}
\def\cG{{\mathcal G}}
\def\cH{{\mathcal H}}
\def\cI{{\mathcal I}}
\def\cJ{{\mathcal J}}
\def\cK{{\mathcal K}}
\def\cL{{\mathcal L}}
\def\cM{{\mathcal M}}
\def\cN{{\mathcal N}}
\def\cO{{\mathcal O}}
\def\cP{{\mathcal P}}
\def\cQ{{\mathcal Q}}
\def\cR{{\mathcal R}}
\def\cS{{\mathcal S}}
\def\cT{{\mathcal T}}
\def\cU{{\mathcal U}}
\def\cV{{\mathcal V}}
\def\cW{{\mathcal W}}
\def\cX{{\mathcal X}}
\def\cY{{\mathcal Y}}
\def\cZ{{\mathcal Z}}
\newcommand{\rmod}[1]{\: \mbox{mod} \: #1}

\def\cg{{\mathcal g}}

\def\e{{\mathbf{\,e}}}
\def\ep{{\mathbf{\,e}}_p}
\def\eq{{\mathbf{\,e}}_q}

\def\Tr{{\mathrm{Tr}}}
\def\Nm{{\mathrm{Nm}}}

\def\rE{{\mathrm{E}}}
\def\rT{{\mathrm{T}}}

 \def\SS{{\mathbf{S}}}

\def\lcm{{\mathrm{lcm}}}

\def\t{\tilde}
\def\ov{\overline}

\def\({\left(}
\def\){\right)}
\def\l|{\left|}
\def\r|{\right|}
\def\fl#1{\left\lfloor#1\right\rfloor}
\def\rf#1{\left\lceil#1\right\rceil}
\def\flq#1{\langle #1 \rangle_q}

\def\mand{\qquad \mbox{and} \qquad}

\newcommand{\commIg}[1]{\marginpar{%
\begin{color}{magenta}
\vskip-\baselineskip %raise the marginpar a bit
\raggedright\footnotesize
\itshape\hrule \smallskip Ig: #1\par\smallskip\hrule\end{color}}}

\newcommand{\commSi}[1]{\marginpar{%
\begin{color}{blue}
\vskip-\baselineskip %raise the marginpar a bit
\raggedright\footnotesize
\itshape\hrule \smallskip Si: #1\par\smallskip\hrule\end{color}}}
\newcommand{\commB}[1]{\marginpar{%
\begin{color}{red}
\vskip-\baselineskip %raise the marginpar a bit
\raggedright\footnotesize
\itshape\hrule \smallskip Br: #1\par\smallskip\hrule\end{color}}}

%%%%%%%%%%%%%%%%%%%%%%%%%%%%%%%%%%%%%%%%%%%%%%%%%%%%%%%%
%%%%%%%%%%%%%%%%%%%%%%%%%%%%%%%%%%%%%%%%%%%%%%%%%%%%%%%%
%%%%%%%%%%%%%%%%%%%%%%%%%%%%%%%%%%%%%%%%%%%%%%%%%%%%%%%%
%%%%%%%%%%%%%%%%%%%%%%%%%%%%%%%%%%%%%%%%%%%%%%%%%%%%%%%%

%%%%%%%  END OF STANDARD STUFF %%%%%%%%%

%%%%%%%%%%%%%%%%%%%%%%%%%%%%%%%%%%%%%%%%%%%%%%%%%%%%%%%%
%%%%%%%%%%%%%%%%%%%%%%%%%%%%%%%%%%%%%%%%%%%%%%%%%%%%%%%%
%%%%%%%%%%%%%%%%%%%%%%%%%%%%%%%%%%%%%%%%%%%%%%%%%%%%%%%%
%%%%%%%%%%%%%%%%%%%%%%%%%%%%%%%%%%%%%%%%%%%%%%%%%%%%%%%
%%%%%%%%%%%
%%% Spell

\hyphenation{re-pub-lished}

\mathsurround=1pt

\def\bfdefault{b}
\overfullrule=5pt

\def \F{{\mathbb F}}
\def \K{{\mathbb K}}
\def \Z{{\mathbb Z}}
\def \Q{{\mathbb Q}}
\def \R{{\mathbb R}}
\def \C{{\\mathbb C}}
\def\Fp{\F_p}
\def \fp{\Fp^*}

\def\Smn{S_{k,\ell,q}(m,n)}

\def\Kmn{\cK_p(m,n)}
\def\psmn{\psi_p(m,n)}

\def\SM{\cS_{k,\ell,q}(\cM)}
\def\SMN{\cS_{k,\ell,q}(\cM,\cN)}
\def\SAMN{\cS_{k,\ell,q}(\cA;\cM,\cN)}
\def\SABMN{\cS_{k,\ell,q}(\cA,\cB;\cM,\cN)}

\def\SIJq{\cS_{k,\ell,q}(\cI,\cJ)}
\def\SAJq{\cS_{k,\ell,q}(\cA;\cJ)}
\def\SABJq{\cS_{k,\ell,q}(\cA, \cB;\cJ)}

\def\sM{\cS_{k,q}^*(\cM)}
\def\sMN{\cS_{k,q}^*(\cM,\cN)}
\def\sAMN{\cS_{k,q}^*(\cA;\cM,\cN)}
\def\sABMN{\cS_{k,q}^*(\cA,\cB;\cM,\cN)}

\def\sIJq{\cS_{k,q}^*(\cI,\cJ)}
\def\sAJq{\cS_{k,q}^*(\cA;\cJ)}
\def\sABJq{\cS_{k,q}^*(\cA, \cB;\cJ)}
\def\sABJp{\cS_{k,p}^*(\cA, \cB;\cJ)}

 \def \xbar{\overline x}

%\begin{titlepage}
%  {\Large Multilinear exponential sums with a general class of weights\par}
%  \vspace{2\baselineskip}
%  Bryce Kerr\par
%  \textit{Department of Pure Mathematics, University of New South Wales,
%Sydney, NSW 2052, Australia} \par
%b.kerr@adfa.edu.au \par
%  \vspace{2\baselineskip}
%    Simon Macourt\par
%  \textit{Department of Pure Mathematics, University of New South Wales,
%Sydney, NSW 2052, Australia} \par
%s.macourt@student.unsw.edu.au \par
%  \vspace{2\baselineskip}
%  \today
%\end{titlepage}

 \author[B.  Kerr] {Bryce Kerr}
\address{Department of Pure Mathematics, University of New South Wales,
Sydney, NSW 2052, Australia}
\email{b.kerr@adfa.edu.au}

 \author[S.  Macourt] {Simon Macourt}
\address{Department of Pure Mathematics, University of New South Wales,
Sydney, NSW 2052, Australia}
\email{s.macourt@unsw.edu.au}

%\begin{abstract}
%
%\end{abstract}

\title[Exponential sums with general weights]{Multilinear exponential sums with a general class of weights}

\begin{abstract}
In this paper we obtain some new estimates for multilinear exponential sums in prime fields with a more general class of weights than previously considered. Our techniques are based on some recent progress of Shkredov in Additive Combinatorics with roots in Rudnev's point plane incidence bound. We apply our estimates to obtain new results concerning exponential sums with sparse polynomials and Weyl sums over small generalized arithmetic progressions.
\end{abstract}

\maketitle

\section{Introduction}

Given a prime number $p$, subsets $\cX_1,\dots,\cX_n \subseteq \F^{*}_p$ and sequences of complex numbers $\omega_1(\textbf{x}),\dots,\omega_n(\textbf{x})$,  we define the weighted multilinear exponential sum over $n$ variables by
\begin{align} \label{eq:S(X1XN)}
S(\cX_1, \dots, \cX_n; \omega_1, \dots, \omega_n)= \sum_{x_1\in\cX_1}\dots \sum_{x_n\in\cX_n} \omega_1(\textbf{x}) \dots \omega_n(\textbf{x}) \ep(x_1\dots x_n),
\end{align}
where $\omega_i$ are $n-1$ dimensional weights that depend on all but the $i$th variable and $\ep(u) =\exp(2 \pi i u/p)$. Assuming each $|\omega_i(\textbf{x})|\le 1$, we are interested in obtaining upper bounds of the form
\begin{align*}
\left|S(\cX_1, \dots, \cX_n; \omega_1, \dots, \omega_n)\right|\le X_1\dots X_np^{-\delta},
\end{align*} 
where $|\cX_i|=X_i$. The first result in this direction is Vinogradov's bilinear estimate and states that 
\begin{align}
\label{eq:Vin}
\left|\sum_{x_1\in X_1}\sum_{x_2\in X_2}\omega_1(x_2)\omega_2(x_1)e_p(x_1x_2) \right|\le p^{1/2}X_1^{1/2}X_2^{1/2},
\end{align}
which is nontrivial provided $X_1X_2>p$. For values of $n\ge 3$  progress has been made through Additive Combinatorics with the first results due to Bourgain, Glibichuck and Konyagin~\cite{BouGliKon} under some restrictions on the sets, weights and number of variables occuring in~\eqref{eq:S(X1XN)}  although their result was general enough to obtain new estimates for sums over small subgroups. Bourgain~\cite{Bour2} extended the results of~\cite{BouGliKon} and obtained an optimal result with respect to the size of $X_1\dots X_n$. In particular, Bourgain showed that for all $\varepsilon>0$ there exists a $\delta>0$ such that
\begin{align*}
\sum_{x_1\in\cX_1}\dots \sum_{x_n\in\cX_n} e_p(x_1\dots x_n)\ll X_1\dots X_np^{-\delta},
\end{align*}
provided
\begin{align*} X_i>p^{\varepsilon}, \quad X_1\dots X_n\ge p^{1+\varepsilon},\end{align*}
and we note that Bourgain gives the dependence of $\delta$ on $\varepsilon$. Recently, Shkredov~\cite{Shkr3} has made significant quantitative improvements to the results of Bourgain by exploiting a  direct connection with geometric incidence estimates of Rudnev~\cite{Rud}. Of particular relevance are the results of Petridis and Shparlinski~\cite{PetShp} and Macourt \cite{Mac1} for recent estimates of three and four dimensional multilinear sums and Shkredov~\cite{Shkr} for the sharpest current results for exponential sums over subgroups of medium size. We mention that a direct application of the methods from~\cite{PetShp,Mac1} are unable to give bounds for multilinear sums beyond four dimensional sums. However, in this paper we are able to break through this barrier and apply related techniques to given new non-trivial results for multilinear sums beyond four variables.  
\newline

Given a set $\cA\subseteq \F_p$ and an integer $k$ we let $D^{\times}_k(\cA)$ count the number of solutions to the equation
\begin{align*}
(a_1-a_2)(a_3-a_4)\dots (a_{2k-1}-a_{2k})=(b_1-b_2)(b_3-b_4)\dots (b_{2k-1}-b_{2k}),
\end{align*}
for $a_i,b_i\in \cA$. The quantity $D_k(\cA)$ plays an important role in our arguments and we obtain some new estimates for $D_k(\cA),$ one of which improves the error term in a result of Shkredov~\cite[Theorem~32]{Shkr3} for sets of cardinality $|\cA|\ge p^{1/2}$. We then apply our estimates to obtain some new bounds for sums of the form~\eqref{eq:S(X1XN)} which are motivated by applications to exponential sums with sparse polynomials and Weyl sums over small generalized arithmetic progressions. 
\newline

%Sums of the form \eqref{eq:S(X1XN)} have been studied for $n=3$ by Bourgain and Garaev \cite{BouGar}. They have then been varied and improved upon since and Petridis and Shparlinski~\cite{PetShp} and Macourt \cite{Mac1} have given bounds for $n=3,4$. More general multilinear sums have been studied by Bourgain (see \cite{Bour1, Bour2, Bour3, Bour4}) and most recently by Shkredov~\cite{Shkr3}. Here we give new bounds for general $n$ and a more general class of weights than previously considered, extending the results of \cite{Mac1, PetShp}. It is clear in estimating such sums that our goal is to obtain something stronger than the trivial bound
%\begin{align*}
%S(\cX_1, \dots, \cX_n;\omega_1, \dots, \omega_n) \le X_1 \dots X_n.
%\end{align*}

Given tuples of integers $a_1,\dots,a_t$ and $k_1,\dots,k_t$ we define the $t$-sparse polynomial
\begin{align} \label{eq:poly123}
\Psi(X) = \sum_{i=1}^t a_iX^{k_i},
\end{align}
and consider the exponential sums
\begin{align} \label{eq:mult}
T_\chi(\Psi) = \sum_{x\in \F^*_p} \chi(x) \ep (\Psi(x)).
\end{align}

Multinomial exponential sums of the form \eqref{eq:mult} have been studied extensively. We first note by the Weil bound, see \cite[Appendix 5, Example 12]{Weil}
\begin{align*}
T_\chi(\Psi) \le p^{1/2} \max(k_1, \dots, k_t).
\end{align*}
 When $k_1,\dots,k_t$ are small the above estimate is sharp and we consider the case when $\max(k_1, \dots, k_t)$ grows with $p$. In this setting, progress on the simplest case of monomials was first made by  Shparlinski \cite{Shp1} and was further improved by Heath-Brown and Konyagin~\cite{HBK} using techniques based on Stepanov's method, although the current sharpest estimates are based on Additive Combinatorics, see for example~\cite{BouGliKon,BouKon}. More general sums of the form~\eqref{eq:mult} were first considered by Mordell~\cite{Mord} and are often referred to as Mordell's exponential sum and we refer the reader to \cite{Bour,CoCoPi1, CoCoPi2, CoPi1, CoPi2, CoPi3, CoPi4} for previous estimates of these sums. We also mention the cases of trinomials and quadrinomials have been given new bounds in \cite{Mac2} and \cite{MSS} and we follow these techniques to reduce to multilinear sums of the form~\eqref{eq:S(X1XN)}. 
\newline

A second application of our bound for the sums~\eqref{eq:S(X1XN)} is a new estimate Weyl sums over small generalized arithmetic progressions. Generalized arithmetic progressions are defined as sets of the form
$$\cA=\{\alpha_1h_1+\dots+\alpha_rh_r+\beta \ : \ 1\le h_i\le H_i\}.$$
For $\cA$ as above, we define the rank of $\cA$ to be $r$ and say that $\cA$ is proper if 
$$|\cA|=H_1\dots H_r.$$
Shao~\cite{Shao1} has previously shown that for any polynomial $F$ we have
\begin{align}
\sum_{a\in \cA}\e_p(F(a))\ll_r p^{1/2+o(1)},
\end{align}
which can be considered a P\'{o}lya-Vinogradov type estimate for generalised arithmetic progressions. We use our estimates for~\eqref{eq:S(X1XN)} to obtain a power saving for Weyl sums over proper generalized arithmetic progressions with an essentially optimal range on the cardinality of $\cA$, see Theorem~~\ref{thm:genap}.
\newline
%Again, our aim is to obtain results that are stronger than the trivial bound
%\begin{align*}
%T_\chi(\Psi) \le p.
%\end{align*}

For the entirety of this paper we let $|\cX_i|=X_i$, and similarly for other sets $|\cY|=Y.$ We also use the notation $A \ll B$ to indicate $A \le c|B|$ for some absolute constant $c$ and similarly $A \ll_k B$ to mean the same where $c$ depends on some parameter $k$.

\subsection{Main Results}
%\begin{theorem} \label{thm:multlin2}
%Let $\cX_i \subset \F_p$ with $|\cX_i|=X_i$, $X_1\ge X_2\ge\dots\ge  X_{n}$, $n\ge 4$ and $ X_{2}^{2-2^{-n+3}}(\log_2(X_{2}))^{-4} < p$ . Then
%\begin{align*}
%&S(\cX_1, \dots, \cX_n; \omega_1, \dots, \omega_n) \\
%&\qquad \ll_n p^{\frac{1}{2^n}}X_1^{\frac{2^n-1}{2^n}}(X_{2}\dots X_{n-1})^{1-\frac{2-2^{-n+3}}{(n-2)2^{n}}+o(1)}X_n^{\frac{2^{n+1}-1}{2^{n+1}}}\\
%&\qquad \qquad \qquad \qquad \qquad \qquad  \qquad + X_1\dots X_{n-2}X_nX_{n-1}^{\frac{2^{n-1}-1}{2^{n-1}}}
%\end{align*}
In what follows we keep notation as in~\eqref{eq:S(X1XN)}.
\begin{theorem} \label{thm:multlin2} 
Let $n\ge 4$,  $\cX_i \subset \F^{*}_p$ subsets satisfying 
\begin{align*}
|\cX_i|=X_i, \quad X_1\ge X_2\ge\dots\ge  X_{n},
\end{align*}
and 
\begin{align*}
X_1X_n^{1/2} \le p.
\end{align*}

 Then we have 
\begin{align*}
&S(\cX_1, \dots, \cX_n; \omega_1, \dots, \omega_n)  \\ &\ll_n X_1\dots X_n \left(\frac{1}{X_1^{1/2}}+\dots+\frac{1}{X_n^{1/2^n}}+p^{\frac{1}{2^n}}X_1^{-\frac{1}{2^n}} X_n^{-\frac{1}{2^{n+1}}}\prod_{i=2}^{n-1} B_{n}(\cX_i)\right)
\end{align*}
where
\begin{align*}
B_n(\cX) = \left\{
\begin{array}{ll}
p^{\frac{1}{2^{2n-3}(n-2)}}X^{-\frac{2^{n-2}+1}{2^{2n-3}(n-2)}+o(1)},& \text{if $ p^{\frac{1}{2}+\frac{1}{2^{n-1}+2}}\ge X \ge p^{\frac{217}{433}}$},\\
X^{-\frac{2^{n-2}-1+2c_1}{2^{2n-3}(n-2)}+o(1)},& \text{if $ p^{\frac{217}{433}} > X\ge p^{\frac{48}{97}}$}, \\
X^{-\frac{2^{n-2}-1+2c_2}{2^{2n-3}(n-2)}+o(1)},& \text{if $ X <p^{\frac{48}{97}}$},
\end{array}
\right.
\end{align*}
and $c_1=\frac{1}{434}$ and $c_2 = \frac{1}{192}$.
\end{theorem}

We give an example of when Theorem \ref{thm:multlin2} is nontrivial. Suppose $n=6$ and $X_1=X_2=\dots =X_6\le p^{\frac{48}{97}}$. Then we have
\begin{align*}
%S(\cX_1, \dots, \cX_6; \omega_1, \dots, \omega_6) \ll p^{\frac{1}{64}}X_1^{\frac{3045}{512}+o(1)}.
S(\cX_1, \dots, \cX_6; \omega_1, \dots, \omega_6) \ll p^{\frac{1}{64}}X_1^{\frac{3110399}{524288}+o(1)}.
\end{align*}
One can see that this is stronger than the trivial bound
\begin{align*}
S(\cX_1, \dots, \cX_6; \omega_1, \dots, \omega_6)\ll X_1^6
\end{align*}
for $X_1>p^{8/27}$. 
%\commB{Removed comparision with Vinogradov's bilinear estimate due to weights}
%We also compare our bound with that of Lemma \ref{lem:bilin} combined with the triangle inequality
%\begin{align*}
%S(\cX_1, \dots, \cX_6; \omega_1, \dots, \omega_6)\ll p^{1/2}X_1^{5}.
%\end{align*}
%One can see our bound is stronger for $X_1 <p^{248/485}$. %We also mention for $n=4$ we recover the bound of \cite[Theorem 1.4]{PetShp} with the addition of a logarithmic factor.
In the case of sets of cardinality a little larger than $p^{1/2}$ we can obtain sharper estimates.
\begin{theorem}
\label{thm:multlin3}
Let $\cX_i \subset \F_p$ satisfy $|\cX_i|=X_i,$ $X_1\ge X_2\ge\dots\ge  X_{n}$ 
\begin{align}
\label{thm:multlin3cond}
|\cX_i|\ge p^{1/2+1/(2^{n+1}-6)}.
\end{align}
Then we have 

\begin{align*}
&|S(\cX_1, \dots, \cX_n; \omega_1, \dots, \omega_n)|\ll_n \\ & X_1\dots X_n\left(\frac{1}{X_1^{1/2}}+\dots+\frac{1}{X_n^{1/2^n}}+p^{o(1)}\left(\frac{p^{1/2}}{(X_1\dots X_n)^{1/n}}\right)^{1/2^n}\right).
\end{align*}

\end{theorem}

The following is a consequence of Theorem~\ref{thm:multlin2}.
%\commB{Im not sure if the conditions involving $p^{\log{p}/2}$ are correct?}
\begin{theorem} \label{thm:multinom} 
Let $\Psi(X)$ be a multinomial of the form \eqref{eq:poly123}, with co-efficients $a_i \in \F^*_p$ for $i=1, \dots , t$. We define
\begin{align*}
\alpha_{k_i}= \gcd(k_i, p-1)
\end{align*}
and
\begin{align*}
\beta_{k_i}= \frac{\alpha_{k_i}}{\gcd(\alpha_{k_i}, \alpha_{k_t})}.
\end{align*}
Suppose $\beta_{k_1}\ge \dots \ge \beta_{k_{t-1}}$. Then
\begin{align*}
&T_\chi(\Psi) \\ & \ll p\left( \left(\frac{\alpha_{k_t}}{p-1}\right)^{\frac{1}{2}}+\beta_{k_1}^{\frac{-1}{2^2}} +\dots+\beta_{k_{t-1}}^{\frac{-1}{2^{t}}}+p^{\frac{1}{2^t}}C_t(\alpha_{k_t}) \prod_{i=1}^{t-2}D_t(\beta_{k_i})\right)
\end{align*}
where
\begin{align*}
C_t(\alpha) = \left\{
\begin{array}{ll}
\alpha^{\frac{3}{2^{t+1}}}p^{-\frac{3}{2^{t+1}}},& \text{if $ \alpha\ge p^{\frac{1}{2}}\log{p}$},\\
\alpha^{\frac{1}{2^{t+1}}}p^{-\frac{1}{2^t}}, &\text{if  $\alpha< p^{\frac{1}{2}}\log{p}$},
\end{array}
\right.
\end{align*}
and
\begin{align*}
D_t(\beta) = \left\{
\begin{array}{ll}
p^{-\frac{1}{2^t(t-2)}},& \text{if $ \beta \ge p^{\frac{1}{2}}\log{p}$},\\
\beta^{-\frac{1}{2^{t-1}(t-2)}}, & \text{if  $\beta< p^{\frac{1}{2}}\log{p}$}.
\end{array}
\right.
\end{align*}
\end{theorem}
We mention that Theorem \ref{thm:multinom} returns the same bound as \cite[Theorem 1.1]{Mac2} when $t=4$. We also mention the strength in this bound is that it relies on mutual greatest common divisors, rather than the size of the exponents. With this in mind, one can give examples of when this is stronger than all known bounds for a given $t$ by first ensuring that $\alpha_{k_t}$ is small and each of the powers are large. We direct the reader to \cite[Corollary 1.2]{Mac2} for such an example for the case $t=4$.

Combining ideas from the proof of Theorem~\ref{thm:multlin2} with estimates of Bourgain for multilinear sums we extend a result of Shao~\cite{Shao1} to the setting of  Weyl sums over small generalized arithmetic progressions.
\begin{theorem}
\label{thm:genap}
Let $p$ be prime, $\cA\subseteq \F_p$ a proper generalized arithmetic progression of rank $r$ and $F\in \F_p[X]$ a polynomial of degree $d$. For any $\varepsilon>0$ there exists some $\delta>0$ such that if 
\begin{align}
\label{eq:genAPcond}
|\cA|\ge p^{1/d+\varepsilon},
\end{align}
then
\begin{align*}
\sum_{a\in \cA}e_p(F(a))\ll_{r,d} |\cA|p^{-\delta}.
\end{align*}
\end{theorem}
We note the condition $|\cA|\ge p^{1/d+\varepsilon}$ is sharp which may be seen by considering the example
\begin{align*}
\cA=\{1,2,\dots,\lfloor p^{1/d}/10 \rfloor\}, \quad F(x)=x^{d},
\end{align*}
so that 
\begin{align*}
\sum_{a\in \cA}e_p(F(a))\gg |\cA|.
\end{align*}
\section{Multilinear Exponential Sums}
\subsection{Reduction  mean values}
\label{sec:mv}
The following result is a variant of \cite[Lemma~2.10]{PetShp} which is more suitable for applications to exponential sums when the variables may run through sets of differing cardinalities.
\begin{lemma} \label{lem:SXin}
Let $n \ge 2$. Suppose $S(\cX_1, \dots, \cX_n; \omega_1, \dots, \omega_n)$ is defined as in \eqref{eq:S(X1XN)}.  Then
\begin{align*}
&|S(\cX_1, \dots, \cX_n; \omega_1, \dots, \omega_n)|^{2^{n-1}}\ll (X_1\dots X_n)^{2^{n-1}}\left(\frac{1}{X_n^{2^{n-2}}}+\dots+\frac{1}{X_2} \right) \\
& \qquad \qquad + X_1^{2^{n-1}-1}(X_2 \dots X_n)^{2^{n-1}-2} \sum_{\substack{x_2,y_2 \in\cX_2 \\ x_2\neq y_2}}\dots \sum_{\substack{x_n, y_n \in\cX_n \\ x_n\neq y_n}} \\
&\qquad \qquad \qquad  \qquad \qquad \times\left|\sum_{x_1 \in \cX_1}\ep(x_1(x_2-y_2)\dots (x_n-y_n))\right|.
\end{align*}
\end{lemma}
\begin{proof}
We proceed by induction on $n$ and first consider the case $n=2$. Our sums take the form
\begin{align*}
S(\cX_1,\cX_2,\omega_1,\omega_2)=\sum_{x_1\in \cX_1}\sum_{x_2\in \cX_2}\omega_1(x_2)\omega_2(x_1)e_p(x_1x_2),
\end{align*}
and hence by the Cauchy-Schwarz inequality 
\begin{align*}
\left|S(\cX_1,\cX_2,\omega_1,\omega_2)\right|^2\le X_1\sum_{x_1\in \cX_1}\left|\sum_{x_2\in \cX_2}e_p(x_1x_2) \right|^2.
\end{align*}
Expanding the square, interchanging summation and isolating the diagonal contribution, we get 
\begin{align*}
\left|S(\cX_1,\cX_2,\omega_1,\omega_2)\right|^2\le X_1^2X_2+X_1 \sum_{\substack{x_2,y_2\in \cX_2 \\ x_2\neq y_2}}\left|\sum_{x_1\in \cX_1}e_p(x_1(x_2-y_2)) \right|.
\end{align*}
Suppose the statement of Lemma~\ref{lem:SXin} is true for some integer $n-1\ge 2$ and consider the sums  $S(\cX_1, \dots, \cX_n; \omega_1, \dots, \omega_n)$. By the Cauchy-Schwarz inequality 
\begin{align*}
& \left|S(\cX_1, \dots, \cX_n; \omega_1, \dots, \omega_n)\right|^2\le X_1\dots X_{n-1} \\ &\sum_{\substack{x_i\in \cX_i \\ 1\le i \le n-1}}\left|\sum_{x_n\in \cX_n}\omega_1(\textbf{x})\dots \omega_{n-1}(\textbf{x})e_p(x_1\dots x_n) \right|^2,
\end{align*}
which after expanding the square, interchanging summation and isolating the diagonal contribution results in
\begin{align*}
\left|S(\cX_1, \dots, \cX_n; \omega_1, \dots, \omega_n)\right|^2 &\le \frac{(X_1\dots X_n)^2}{X_n} \\ &+X_1\dots X_{n-1}\sum_{\substack{x_n,y_n\in \cX_n \\ x_n\neq y_n}}S(x_n,y_n),
\end{align*}
where 
\begin{align*}
&S(x_n,y_n)= \\ & \left|\sum_{\substack{x_i\in \cX_i \\ 1\le i \le n-1}}\omega_1'(\mathbf{x}',x_n,y_n)\dots \omega_{n-1}'(\mathbf{x}',x_n,y_n)e_p(x_1\dots x_{n-1}(x_n-y_n)) \right|,
\end{align*}
and 
$$\textbf{x}'=(x_1,\dots,x_{n-1}), \quad \omega_j(\textbf{x}',x_n,y_n)=\omega_j(\textbf{x}',x_n)\overline \omega_j(\textbf{x}',y_n).$$
By H\"{o}lder's inequality
\begin{align*}
&\left|S(\cX_1, \dots, \cX_n; \omega_1, \dots, \omega_n)\right|^{2^{n-1}}\ll \frac{(X_1\dots X_n)^{2^{n-1}}}{X_n^{2^{n-2}}} \\ 
&\quad +(X_1\dots X_{n-1})^{2^{n-2}}X_n^{2^{n-1}-2}\sum_{\substack{x_n,y_n\in \cX_n \\ x_n\neq y_n}}S(x_n,y_n)^{2^{n-2}}.
\end{align*}
We next fix some pair $x_n\neq y_n$ and apply our induction hypothesis to the sum $S(x_n,y_n)$. This gives
\begin{align*}
& S(x_n,y_n)^{2^{n-2}}\le (X_1\dots X_{n-1})^{2^{n-2}}\left(\frac{1}{X_{n-1}^{2^{n-3}}}+\dots+\frac{1}{X_2} \right) \\ &+X_1^{2^{n-2}-1}(X_2\dots X_{n-1})^{2^{n-2}-2} \\ &\sum_{\substack{x_i,y_i\in \cX_i \\ x_i\neq y_i \\ 2\le i \le n-1}}\left|\sum_{x_1\in \cX_1}e_p(x_1(x_2-y_2)\dots(x_{n-1}-y_{n-1})(x_n-y_n)) \right|,
\end{align*}
which combined with the above implies
\begin{align*}
&\left|S(\cX_1, \dots, \cX_n; \omega_1, \dots, \omega_n)\right|^{2^{n-1}}\le (X_1\dots X_n)^{2^{n-1}}\left(\frac{1}{X_n^{2^{n-2}}}+\dots+\frac{1}{X_2} \right) \\&+X_1^{2^{n-1}-1}(X_2\dots X_n)^{2^{n-1}-2} \\ & \times \sum_{\substack{x_i,y_i\in \cX_i \\ x_i\neq y_i \\ 2\le i \le n-1}}\left|\sum_{x_1\in \cX_1}e_p(x_1(x_2-y_2)\dots(x_{n-1}-y_{n-1})(x_n-y_n)) \right|,
\end{align*}
and completes the proof.
\end{proof}

 We mention that the above proof is independent of the sizes of the $X_i$, and as such the lemma is left without such restrictions. 
 
For any set $\cA \subset \F_p$ we define
 \begin{align*}
 &D^\times_k(\cA) \\
 &= |\{(a_1-a_2)\dots(a_{2k-1}-a_{2k}) = (b_1-b_2)\dots(b_{2k-1}-b_{2k}) : a_i, b_i \in \cA\}|,
 \end{align*}
and extend the notation when variables run through different sets by defining
$D^\times_k(\cX_1, \dots, \cX_k)$ to be the number of solutions to
 \begin{align*} 
 (w_1-x_1)\dots(w_{k}-x_{k}) = (y_1-z_1)\dots(y_{k}-z_{k}),
\end{align*}
for $w_i, x_i, y_i, z_i \in \cX_i$. Finally, we use the notation $D_k^{\times, *}$ for the above cases where we exclude the solutions when the equation is $0$ and  define
\begin{align*}
\widetilde D_k^{\times, *}(\cX_1,\dots,\cX_k)=D_k^{\times, *}(\cX_1,\dots,\cX_k)-\frac{\left(\prod_{i=1}^kX_i(X_i-1) \right)^2}{p-1}.
\end{align*}
We note that $\widetilde D_k^{\times, *}$ is the error in approximation of $D_k^{\times, *}$ by the expected main term.
\begin{lemma} \label{lem:Dktimes*}
Let $\cX_1, \dots ,\cX_k \subset \F_p$. Then
\begin{align*}
D^{\times,*}_k(\cX_1, \dots, \cX_k) \le (D^{\times,*}_k(\cX_1)\dots  D^{\times,*}_k(\cX_k))^{1/k}.
\end{align*}
\end{lemma}
\begin{proof}
We let $K=D^{\times,*}_k(\cX_1, \dots, \cX_k)$ and express $K$ in terms of multiplicative characters
\begin{align*}
K&= \sum_{w_1,x_1,y_1,z_1 \in \cX_1} \dots \sum_{w_k,x_k,y_k,z_k \in \cX_k} \\
&\qquad \quad \frac{1}{p-1}\sum_{\chi \in \Omega} \chi(w_1-x_1)\dots(w_{k}-x_{k}) \overline{\chi}(y_1-z_1)\dots(y_{k}-z_{k})
\end{align*}
where $\Omega$ is the set of all distinct characters. Clearly,
\begin{align*}
K = \frac{1}{p-1}\sum_{\chi \in \Omega}\left| \sum_{w_1,x_1\in \cX_1} \chi(w_1-x_1)\right|^2 \dots \left|\sum_{w_k,x_k\in \cX_k} \chi(w_k-x_k)\right|^2.
\end{align*}
Using Holder's inequality, we obtain
\begin{align*}
K^k &\le \frac{1}{(p-1)^k} \sum_{\chi \in \Omega}\left| \sum_{w_1,x_1\in \cX_1} \chi(w_1-x_1)\right|^{2k} \dots \sum_{\chi \in \Omega}\left| \sum_{w_k,x_k\in \cX_k} \chi(w_k-x_k)\right|^{2k}\\
&= D^{\times,*}_k(\cX_1)\dots  D^{\times,*}_k(\cX_k).
\end{align*}
\end{proof}
The proof of the following is similar to that of Lemma~\ref{lem:Dktimes*} with summation only over non-principal characters.

\begin{lemma} \label{lem:tildeDktimes*}
Let $\cX_1, \dots ,\cX_k \subset \F_p$. Then
\begin{align*}
\widetilde D^{\times,*}_k(\cX_1, \dots, \cX_k) \le (\widetilde D^{\times,*}_k(\cX_1)\dots  \widetilde D^{\times,*}_k(\cX_k))^{1/k}.
\end{align*}
\end{lemma}

Using Lemma~\ref{lem:SXin}, Lemma~\ref{lem:Dktimes*} and Lemma~\ref{lem:tildeDktimes*} we give two general results relating estimates for $S(\cX_1, \dots, \cX_n; \omega_1, \dots, \omega_n)$ to the quantities $D^\times_k(\cA)$ and $\widetilde D^\times_k(\cA)$. We first recall the classic Vinogradov bilinear estimate,  see \cite[Equation 1.4]{BouGar} or \cite[Lemma 4.1]{Gar}.
\begin{lemma}
\label{lem:bilin} 
For any sets $\cX, \cY \subseteq \F_p$ and any  $\alpha= (\alpha_{x})_{x\in \cX}$, $\beta = \( \beta_{y}\)_{y \in \cY}$ with 
\begin{align*}
\sum_{x\in \cX}|\alpha_{x}|^2 = A \mand  \sum_{y \in \cY}|\beta_{y}|^2 = B, 
\end{align*}
we have 
\begin{align*}
\left |\sum_{x \in \cX}\sum_{y \in \cY} \alpha_{x} \beta_{y}  \ep(xy) \right| \le \sqrt{pAB}.
\end{align*}
\end{lemma}
\begin{lemma}
\label{lem:SMV}
Let $n \ge 2$. Suppose $S(\cX_1, \dots, \cX_n; \omega_1, \dots, \omega_n)$ is defined as in \eqref{eq:S(X1XN)} and that 
$$X_1\ge X_2\dots \ge X_n.$$
Then
\begin{align*}
&|S(\cX_1, \dots, \cX_n; \omega_1, \dots, \omega_n)|^{2^{n}}\ll (X_1\dots X_n)^{2^{n}}\left(\frac{1}{X_1^{2^{n-1}}}+\dots+\frac{1}{X_{n-1}^2} \right) \\
& \qquad \qquad + pX_n^{2^{n}-1}(X_1 \dots X_{n-1})^{2^{n}-4}(D^{\times,*}_{n-1}(\cX_1)\dots  D^{\times,*}_{n-1}(\cX_{n-1}))^{1/(n-1)}.
\end{align*}
\end{lemma}
\begin{proof}
Writing
$$S=\sum_{\substack{x_1,y_1 \in\cX_1 \\ x_1\neq y_1}}\dots \sum_{\substack{x_{n-1}, y_{n-1} \in\cX_{n-1} \\ x_{n-1}\neq y_{n-1}}}\left|\sum_{x_n \in \cX_n}\ep(x_n(x_1-y_1)\dots (x_{n-1}-y_{n-1}))\right|,$$
by Lemma~\ref{lem:SXin} it is sufficient to show that
\begin{align*}
&S^2\le    pX_n(D^{\times,*}_{n-1}(\cX_1)\dots  D^{\times,*}_{n-1}(\cX_{n-1}))^{1/(n-1)}.
\end{align*} 
Let $I(\lambda)$ count the number of solutions to the equation
$$\lambda=(x_1-y_1)\dots(x_{n-1}-y_{n-1}), \quad x_i,y_i\in \cX_i, \ \ x_i\neq y_i,$$
so that 
\begin{align*}
S=\sum_{\lambda}I(\lambda)\left|\sum_{x_n\in \cX_n}e_p(\lambda x_1) \right|,
\end{align*}
and hence by Lemma~\ref{lem:bilin}
\begin{align*}
S^2\le \left(\sum_{\lambda}I(\lambda)^2\right)pX_n,
\end{align*}
and the result follows from Lemma~\ref{lem:Dktimes*} since 
\begin{align*}
\sum_{\lambda}I(\lambda)^2=D^\times_{n-1}(\cX_1, \dots, \cX_{n-1}).
\end{align*}
\end{proof}
Our next estimate does better in applications over Lemma~\ref{lem:SMV} when our sets have $\cX_1,\dots,\cX_n$ have large cardinalities.
\begin{lemma} 
\label{lem:SMV1} %\commSi{Removed the condition that $X_1\ge X_2 \dots $ as we don't seem to need it at any point in this proof}
Let $n \ge 2$. Suppose $S(\cX_1, \dots, \cX_n; \omega_1, \dots, \omega_n)$ is defined as in \eqref{eq:S(X1XN)}. Then we have 
\begin{align*}
&|S(\cX_1, \dots, \cX_n; \omega_1, \dots, \omega_n)|^{2^{n}}\ll (X_1\dots X_n)^{2^{n}}\left(\frac{1}{X_1^{2^{n-1}}}+\dots+\frac{1}{X_n} \right) \\
& \qquad \qquad +p^{1/2}(X_1\dots X_n)^{2^n-2}(\widetilde D^{\times,*}_n(\cX_1)\dots  \widetilde D^{\times,*}_n(\cX_n))^{1/2n}.
\end{align*}
\begin{proof}
Writing
$$S=\sum_{\substack{x_2,y_2 \in\cX_2 \\ x_2\neq y_2}}\dots \sum_{\substack{x_n, y_n \in\cX_n \\ x_n\neq y_n}}\left|\sum_{x_1 \in \cX_1}\ep(x_1(x_2-y_2)\dots (x_n-y_n))\right|,$$
by Lemma~\ref{lem:SXin} it is sufficient to show that
\begin{align*}
&S^2\le \frac{(X_1\dots X_n)^4}{X_1^2} \\ & \quad \quad \quad +(X_2\dots X_n)^2p^{1/2}(\widetilde D^{\times,*}_n(\cX_1)\dots  \widetilde D^{\times,*}_n(\cX_n))^{1/2n}.
\end{align*}
Applying the Cauchy-Schwarz inequality, interchanging summation and isolating the diagonal contribution gives
\begin{align}
\label{eq:SMVS2}
S^2\le X_1(X_2\dots X_n)^4+(X_2\dots X_n)^2\left|\sum_{\lambda=1}^{p-1}I(\lambda)e_p(\lambda)\right|,
\end{align}
where $I(\lambda)$ counts the number of solutions to the equation
\begin{align*}
(x_1-y_1)\dots(x_n-y_n)=\lambda, \quad x_i,y_i\in \cX_i, \ \ x_i\neq y_i.
\end{align*}
Let 
$$\Delta=\frac{X_1(X_1-1)\dots X_n(X_n-1)}{p-1},$$
and write
\begin{align*}
\sum_{\lambda=1}^{p-1}I(\lambda)e_p(\lambda)=\Delta\sum_{\lambda=1}^{p-1}e_p(\lambda)+\sum_{\lambda=1}^{p-1}(I(\lambda)-\Delta)e_p(\lambda).
\end{align*}
We have 
\begin{align}
\label{eq:SMVS2-1}
\left|\sum_{\lambda=1}^{p-1}I(\lambda)e_p(\lambda)\right|\ll \frac{(X_1\dots X_n)^2}{p}+\sum_{\lambda=1}^{p-1}|I(\lambda)-\Delta|.
\end{align}
With notation as in Lemma~\ref{lem:tildeDktimes*}, by the Cauchy-Schwarz inequality 
\begin{align*}
\sum_{\lambda=1}^{p-1}|I(\lambda)-\Delta|\le p^{1/2}\left(\sum_{\lambda=1}^{p-1}|I(\lambda)-\Delta|^2\right)^{1/2}=p^{1/2}\widetilde D^{\times,*}_n(\cX_1, \dots, \cX_n)^{1/2},
\end{align*}
and hence 
\begin{align*}
\sum_{\lambda=1}^{p-1}|I(\lambda)-\Delta|\le p^{1/2}(\widetilde D^{\times,*}_n(\cX_1)\dots  \widetilde D^{\times,*}_n(\cX_n))^{1/2n}.
\end{align*}
Combining the above with~\eqref{eq:SMVS2} and~\eqref{eq:SMVS2-1} gives
\begin{align*}
S^2&\le \frac{(X_1\dots X_n)^4}{X_1^3}+\frac{(X_1\dots X_n)^4}{p}\\
&\qquad \qquad +(X_2\dots X_n)^2p^{1/2}(\widetilde D^{\times,*}_n(\cX_1)\dots  \widetilde D^{\times,*}_n(\cX_n))^{1/2n} \\ 
&\ll \frac{(X_1\dots X_n)^4}{X_1^3}+(X_2\dots X_n)^2p^{1/2}(\widetilde D^{\times,*}_n(\cX_1)\dots  \widetilde D^{\times,*}_n(\cX_n))^{1/2n},
\end{align*}
and completes the proof.
\end{proof}
\end{lemma}
\subsection{Estimates for $D^{\times}_k(\cA)$}
In this section we give estimates for $D^{\times}_k(\cA)$ which will be combined with results from Section~\ref{sec:mv} to obtain estimates for multilinear sums. We first recall the following result \cite[Theorem 32]{Shkr3}.
\begin{lemma} \label{lem:DtimesE+}
Suppose $\cA \subset \F_p$  is a set and $|\cA|=A$. For all $k\ge 2$
\begin{align*}
D^\times_k(\cA) - \frac{A^{4k}}{p} \ll_k (\log A)^4 A^{4k-2-2^{-k+2}}E^+(\cA)^{1/2^{k-1}}.
\end{align*}
\end{lemma}
We then have the following lemma \cite[Theorem 41]{Shkr3}.
\begin{lemma} \label{lem:Dtimessmall}
Let $\cA \subset \F_p$ be a set, $A \le p^{2846/4991}$. Then for any $c < \frac{1}{434}$ one has 
\begin{align*}
D_2^\times(\cA) \ll A^{13/2-c}.
\end{align*}
Furthermore, if $A \le p^{48/97}$, then for any $c_1 < \frac{1}{192}$ one has
\begin{align*}
D_2^\times(\cA) \ll A^{13/2-c_1}.
\end{align*}
\end{lemma}
We first notice that from the proof of \cite[Theorem 32]{Shkr3} we have 
\begin{align} \label{eq:Dkk-1}
D^\times_k(\cA)-\frac{A^{4k}}{p} \ll_k (\log A)^2 A^{2k+1}\left(D^\times_{k-1}(A) -\frac{A^{4(k-1)}}{p}\right)^{1/2}.
\end{align}
Using $E^+(\cA) \le A^3$, combined with Lemma \ref{lem:Dtimessmall} and \eqref{eq:Dkk-1} we have the following corollary.
\begin{cor}\label{cor:Dktimes}
Suppose $\cA \subset \F_p$  is a set and $|\cA|=A$. For all $k\ge 2$
\begin{align*}
D^\times_k(\cA) - \frac{A^{4k}}{p} \ll_k (\log A)^4 A^{4k-2+2^{-k+1}}.
\end{align*}
Similarly if $A\le p^{2846/4991}$, for any $c<\frac{1}{434}$ we have
\begin{align*}
D^\times_k(\cA) - \frac{A^{4k}}{p} \ll_k (\log A )^4 A^{4k-2+2^{-k+1}-c2^{-k+2}}
\end{align*}
and if $A \le p^{48/97}$, for any $c_1 < \frac{1}{192}$ we have
\begin{align*}
D^\times_k(\cA) - \frac{A^{4k}}{p} \ll_k (\log A )^4 A^{4k-2+2^{-k+1}-c_12^{-k+2}}.
\end{align*}
\end{cor}
It is clear that we can use the above to give other estimates on $D^\times_k$ using previous estimates on $D^\times_2$. We recall the following result \cite[Lemma 2.6]{Mac1}, which is given from Murphy et. al \cite{MPR-NRS} result on collinear triples.
\begin{lemma}
Let $\cA \subset \F_p$. Then
\begin{align*}
D^\times_2(\cA)- \frac{A^8}{p} \ll p^{1/2}A^{11/2}.
\end{align*}
\end{lemma}
Again, we have the following corollary.
\begin{cor} \label{cor:Dktimes2}
Let $\cA \subset \F_p$. Then
\begin{align*}
D^\times_k(\cA) - \frac{A^{4k}}{p} \ll_k p^{2^{1-k}}(\log A)^4 A^{4k-2-2^{-k+1}}.
\end{align*}
\end{cor}

We next prepare to give an estimate for $D^{\times}_k(\cA)$ which improves on the above results for sets of cardinality a little larger than $p^{1/2}$. As in Shkredov~\cite{Shkr3}, our main tool is Rudnev's point plane incidence bound~\cite{Rud}.
\begin{lemma}
\label{lem:Rud}
Let $p$ be an odd prime, $\cP\subset \F_p^3$ a set of points and $\Pi$ a collection of planes in $\F_p^3$. Suppose $|\cP|\le |\Pi|$ and that $k$ is the maximum number of collinear points in $\cP$. Then the number of point-planes incidences satisfies
$$\cI(\cP,\Pi)\le \frac{|\cP||\Pi|}{p}+|\cP|^{1/2}|\Pi|+k|\cP|.$$
\end{lemma}

\begin{lemma}
\label{lem:Dtimes}
For a prime number $p$ and a subset $\cA\subseteq \F_p$ with $|\cA|=A$ we have
%we let 
%$D^{\times}(\cA)$ count the number of solutions to the equation
%$$(a_1-a_2)(a_3-a_4)= (a_5-a_6)(a_7-a_8),$$
%with variables $a_1,\dots,a_8\in \cA$. Then we have 
\begin{align*}
D_2^{\times}(\cA)&=\frac{A^8}{p}+O\left(A^6(\log{A})^2+p^{1/2}A^4E_{+}(\cA)^{1/2}(\log{A})^2\right) \\ & \quad \quad \quad+O\left(pA^4(\log{A})^2\right).
\end{align*}
\end{lemma}
\begin{proof}
We have 
\begin{align*}
D_2^{\times}(\cA)=\sum_{\substack{a_i\in \cA \\ (a_1-a_2)(a_3-a_4)=(a_5-a_6)(a_7-a_8) \\ a_5\neq a_6 }}1+O(A^6).
\end{align*}
Let $I(x)$ denote the indicator function of the multiset $$\{ a-a' \ : \ a,a'\in \cA\},$$  and let $\widehat I$ denote the Fourier transform of $I$. We note that the Fourier coefficients satisfy
\begin{align}
\label{eq:Ihat}
\widehat I(x)=\left|\sum_{a\in \cA}e_p(ax) \right|^2.
\end{align}
We have 
\begin{align}
\label{eq:DW}
\nonumber D_2^{\times}(A)&=\sum_{\substack{a_i\in \cA  \\ a_5\neq a_6 }}I\left(\frac{(a_1-a_2)(a_3-a_4)}{(a_5-a_6)} \right)+O(A^6) \\
&=\frac{A^8}{p}+O(A^6)+W,
\end{align}
where 
\begin{align*}
W=\frac{1}{p}\sum_{y=1}^{p-1}\widehat I(y)\sum_{\substack{a_i \in \cA \\ a_5\neq a_6}}e_p(-y(a_1-a_2)(a_3-a_4)(a_5-a_6)^{-1}).
\end{align*}
We have 
\begin{align*}
W&\le \frac{1}{p}\sum_{y=1}^{p-1}\sum_{z=1}^{p}\widehat I(y) \widehat I(z)\sum_{\substack{a_i \in \cA \\ (a_1-a_2)y=(a_3-a_4)z \\ a_3\neq a_4}}1 \\
&=\frac{A^5}{p}\sum_{y=1}^{p-1}\widehat I(y)+\frac{1}{p}\sum_{y=1}^{p-1}\sum_{z=1}^{p-1}\widehat I(y) \widehat I(z)\sum_{\substack{a_i \in \cA \\ (a_1-a_2)y=(a_3-a_4)z }}1, 
\end{align*}
where we have removed the condition $a_3\neq a_4$ in the last display since by~\eqref{eq:Ihat} the Fourier coefficients are nonnegative. The above implies 
\begin{align}
\label{eq:WW0}
W\le W_0+O(A^6),
\end{align}
where
\begin{align*}
W_0=\frac{1}{p}\sum_{y=1}^{p-1}\sum_{z=1}^{p-1}\widehat I(y) \widehat I(z)\sum_{\substack{a_i \in \cA \\ (a_1-a_2)y=(a_3-a_4)z}}1.
\end{align*}
For integer $i\ge 1$ we define the sets
\begin{align}
\label{eq:Jidef}
J(i)=\{ 1\le z \le p \ : \ 2^{i-1}-1\le \widehat I(z)< 2^{i}-1 \},
\end{align}
so that 
\begin{align}
\label{eq:WWj}
W_0\ll \frac{1}{p}\sum_{1\le i,j \ll \log{A}}2^{i+j}W(i,j),
\end{align}
where 
\begin{align*}
W(i,j)=\sum_{\substack{a_i \in \cA, y\in J(i), z\in J(j) \\  (a_1-a_2)y=(a_3-a_4)z}}1.
\end{align*}
Fix some pair $(i,j)$ and consider $W(i,j)$. If 
$|J(i)|\le |J(j)|,$
then we consider the set of points 
$$\cP=\{ (a_1y,y,a_3) \ : \ y\in J(i), \  a_1,a_3\in \cA \},$$
and the collection of planes  
$$\Pi = \{ x_1-a_2x_2-zx_3+a_4z=0 \ : \  z\in J(j), \ \ a_2,a_4\in \cA \}.$$
We see that $W(i,j)$ is bounded by the number of point-plane incidences between $\cP$ and $\Pi$
\begin{align*}
W(i,j)\le \cI(\cP,\Pi).
\end{align*}
Since the maximum number of collinear points in $\cP$ is $\max\{ A,|J(i)|\}$ an  application of Lemma~\ref{lem:Rud} gives
\begin{align}
\begin{split}
\label{eq:Wij1}
W(i,j)\ll \frac{A^4|J(i)||J(j)|}{p}&+A^3|J(i)|^{1/2}|J(j)|\\
&\qquad \qquad +A^2|J(i)|\max\{ A,|J(i)|\}.
\end{split}
\end{align}
In a similar fashion, if $|J(j)|\le |J(i)|$ then
\begin{align}
\begin{split}
\label{eq:Wij2}
W(i,j)\ll \frac{A^4|J(i)||J(j)|}{p}&+A^3|J(j)|^{1/2}|J(i)|\\
& \qquad \qquad +A^2|J(j)|\max\{ A,|J(j)|\}.
\end{split}
\end{align}
 This implies that
\begin{align*}
W(i,j) &\ll \frac{A^4|J(i)||J(j)|}{p}+A^3|J(i)|^{1/2}|J(j)|+A^3|J(j)|^{1/2}|J(i)| \\ & \quad \quad +A^2\min\{|J(i)|^2,|J(j)|^2\} \\ & \ll \frac{A^4|J(i)||J(j)|}{p}+A^3|J(i)|^{1/2}|J(j)|+A^3|J(j)|^{1/2}|J(i)| \\ & \quad \quad +A^2|J(i)||J(j)|,
\end{align*}
and hence substituting the above into~\eqref{eq:WWj} we get 
\begin{align*}
W_0 &\ll \frac{A^4}{p^2}\left(\sum_{1\le i \ll \log{A}}2^{i}|J(i)| \right)^2 \\  
& \qquad +\frac{A^3}{p}\left(\sum_{1\le i \ll \log{A}}2^i|J(i)|^{1/2} \right)\left(\sum_{1\le i \ll \log{A}}2^{i}|J(i)| \right)\\
&\qquad \qquad +\frac{A^2}{p}\left(\sum_{1\le i \ll \log{A}}2^{i}|J(i)| \right)^2.
\end{align*}
Recalling~\eqref{eq:Ihat} and~\eqref{eq:Jidef}, we have 
\begin{align*}
\sum_{1\le i \ll \log{A}}2^{i}|J(i)|&\ll p+\sum_{2\le i \ll \log{A}}2^{i}|J(i)|  \\ & \ll p+\log{A}\sum_{y=1}^{p}|\sum_{a\in \cA}e_p(ya)|^2=pA\log{A},
\end{align*}
and 
\begin{align*}
\left(\sum_{1\le i \ll \log{A}}2^i|J(i)|^{1/2} \right)^2&\ll p+\log{A}\sum_{2\le i \ll \log{A}}2^{2i}|J(i)|\\
&\ll p+(\log{A})^2\sum_{y=1}^{p}\left|\sum_{a\in \cA}e_p(ya) \right|^4,
\end{align*}
so that 
\begin{align*}
\sum_{1\le i \ll \log{A}}2^i|J(i)|^{1/2}\ll p^{1/2}E_{+}(\cA)^{1/2}\log{A}.
\end{align*}
This implies
\begin{align*}
W\ll A^6(\log{A})^2+p^{1/2}A^4E_{+}(\cA)^{1/2}(\log{A})^2+pA^4,
\end{align*}
and hence by~\eqref{eq:DW} and~\eqref{eq:WW0}
\begin{align*}
D_2^{\times}(A)&=\frac{A^8}{p}+O\left(A^6(\log{A})^2\right)+O\left(p^{1/2}A^4E_{+}(\cA)^{1/2}(\log{A})^2\right) \\ &+O(pA^4(\log{A})^2),
\end{align*}
which completes the proof.
\end{proof}
We next establish a recurrence type inequality similar to~\cite[Theorem~32]{Shkr3}. 
\begin{lemma}
\label{lem:Dktimes}
For a prime number $p$ and a subset $\cA\subseteq \F_p$ with $|\cA|=A$ 
%we let 
%$D_k^{\times}(\cA)$ count the number of solutions to the equation
%$$(a_{1,1}-a_{1,2})\dots (a_{k,1}-a_{k,2})= (a_{k+1,1}-a_{k+1,2})\dots(a_{2k,1}-a_{2k,2}),$$
%with variables $a_{1,1},\dots,a_{2k,2}\in \cA$. Then we have 
we have
\begin{align*}
D_k^{\times}(\cA)=\frac{A^{4k}}{p}+O_k\left(\left(A^{4k-2}+pA^{4k-4}+p^{1/2}A^{2k}D^{\times}_{k-1}(\cA)^{1/2}\right)\log^2{A}\right).
\end{align*}
\end{lemma}
\begin{proof}
Let $D'_k(\cA)$ count the number of solutions to the equation 
$$(a_{1,1}-a_{1,2})\dots (a_{k,1}-a_{k,2})= (a_{k+1,1}-a_{k+1,2})\dots(a_{2k,1}-a_{2k,2}),$$
with variables $a_{1,1},\dots,a_{2k,2}\in \cA$  satisfying
$$a_{1,1}\neq a_{1,2}, \quad a_{k+1,1}\neq a_{k+1,2},$$
so that 
\begin{align}
\label{eq:DD'}
D^{\times}_k(\cA)=D'_k(\cA)+O(A^{4k-2}).
\end{align}
Let $I(y)$ denote the indicator function of the multiset 
$$\{ (a_{2,1}-a_{2,2})\dots (a_{k,1}-a_{k,2}) \ : a_{2,1},\dots,a_{k,2}\in \cA\},$$
and let $\widehat I(y)$ denote the Fourier transform of $I$. We have
\begin{align*}
D'_k&(\cA)=\sum_{\substack{a_{j,1},a_{j,2}\in \cA \\ a_{1,1}\neq a_{1,2}\\ a_{k+1,1}\neq a_{k+1,2}}}I((a_{k+1,1}-a_{k+1,2})\dots(a_{2k,1}-a_{2k,2})(a_{1,1}-a_{1,2})^{-1}) \\ 
&=\frac{1}{p}\sum_{y=1}^{p-1}\widehat I(y)\\
& \sum_{\substack{a_{j,1},a_{j,2}\in \cA \\ a_{1,1}\neq a_{1,2}\\ a_{k+1,1}\neq a_{k+1,2}}}e_p\left(-y(a_{k+1,1}-a_{k+1,2})\dots(a_{2k,1}-a_{2k,2})(a_{1,1}-a_{1,2})^{-1} \right) \\
&=\frac{1}{p}\sum_{z=1}^p\sum_{y=1}^{p-1}\widehat I(y)\widehat I(-z)\sum_{\substack{a_{i,j}\in \cA \\ y(a_{1,1}-a_{1,2})=z(a_{2,1}-a_{2,2}) \\ a_{j,1}\neq a_{j,2}, \ j=1,2}}1,
\end{align*}
which implies that 
\begin{align}
\label{eq:DkW0}
D'_k(\cA)=\frac{A^{4k}}{p}+W_0+O(A^{4k-2}),
\end{align}
where 
$$W_0=\frac{1}{p}\sum_{z=1}^{p-1}\sum_{y=1}^{p-1}\widehat I(y)\widehat I(-z)\sum_{\substack{a_{i,j}\in \cA \\ y(a_{1,1}-a_{1,2})=z(a_{2,1}-a_{2,2}) \\ a_{j,1}\neq a_{j,2}, \ j=1,2}}1.$$
For integer $i\ge 1$ we define 
$$J(i)=\{ y\in \F_p^* \ : \ 2^{i-1}-1\le |\widehat I(y)|\le 2^i-1  \},$$
so that 
\begin{align}
\label{eq:W0d1}
W_0\ll \frac{1}{p}\sum_{\substack{i,j \ll \log{A^{2k}}}}2^{i+j}W(i,j),
\end{align}
where 
\begin{align*}
W(i,j)=\sum_{\substack{a_{i,j}\in \cA, \\ y\in J(i), z\in J(j)  \\ y(a_{1,1}-a_{1,2})=z(a_{2,1}-a_{2,2}) \\ a_{j,1}\neq a_{j,2}, \ j=1,2}}1.
\end{align*}
Using Lemma~\ref{lem:Rud} as in the proof of Lemma~\ref{lem:Dtimes}, we see that 
\begin{align}
\label{eq:Wij123123}
W(i,j) & \ll \frac{A^4|J(i)||J(j)|}{p}+A^3|J(i)|^{1/2}|J(j)|+A^3|J(j)|^{1/2}|J(i)| \\ & \quad \quad +A^2|J(i)||J(j)| \nonumber.
\end{align}
We have 
\begin{align*}
&\sum_{i\ll \log{A}}2^{i}|J(i)|\\
&\qquad \ll p+\sum_{y=1}^{p-1}\left|\sum_{\substack{a_{i,1},a_{i,2}\in \cA \\ 1\le i \le k-1}}e_p(y(a_{1,1}-a_{1,2})\dots (a_{k-1,1}-a_{k-1,2})) \right| \\
&\qquad \le p+\sum_{\substack{a_{i,1},a_{i,2}\in \cA \\ 2\le i \le k-1}}\sum_{y=1}^{p-1}\left|\sum_{a\in A}e_p(y(a_{2,1}-a_{2,2})\dots (a_{k-1,1}-a_{k-1,2})a) \right|^2 \\ 
&\qquad \ll pA^{2k-3},
\end{align*}
and 
\begin{align*}
& \sum_{i\ll \log{A}}2^{i}|J(i)|^{1/2}\ll  p^{1/2}\\ &+\left( \log{A}\sum_{y=1}^{p-1}\left|\sum_{\substack{a_{i,1},a_{i,2}\in \cA \\ 1\le i \le k-1}}e_p(y(a_{1,1}-a_{1,2})\dots (a_{k-1,1}-a_{k-1,2})) \right|^2 \right)^{1/2},
\end{align*}
so that 
\begin{align*}
\sum_{i\ll \log{A}}2^{i}|J(i)|^{1/2}\ll_k (\log{A})^{1/2}p^{1/2}D^{\times}_{k-1}(\cA)^{1/2}.
\end{align*}
Combining the above with~\eqref{eq:W0d1} and~\eqref{eq:Wij123123} we see that
\begin{align*}
W_0\ll_k \left(A^{4k-2}+pA^{4k-4}+p^{1/2}A^{2k}D^{\times}_{k-1}(\cA)^{1/2}\right)\log^2{A},
\end{align*}
and hence by~\eqref{eq:DD'} and~\eqref{eq:DkW0}
\begin{align*}
D_k^{\times}(A)&=\frac{A^{4k}}{p}\\
&\quad +O_k\left(\left(A^{4k-2}+pA^{4k-4}+p^{1/2}A^{2k}D^{\times}_{k-1}(\cA)^{1/2}\right)\log^2{A}\right),
\end{align*}
which completes the proof.
\end{proof}
Combining Lemma~\ref{lem:Dtimes} and Lemma~\ref{lem:Dktimes} with an induction argument gives the following Corollary.
\begin{cor}
\label{cor:Dk1}
For a prime number $p$ and a subset $\cA\subseteq \F_p$ with $|\cA|=A \ge p^{1/2}$ we have 
%we let 
%$D_k^{\times}(\cA)$ count the number of solutions to the equation
%$$(a_{1,1}-a_{1,2})\dots (a_{k,1}-a_{k,2})= (a_{k+1,1}-a_{k+1,2})\dots(a_{2k,1}-a_{2k,2}),$$
%with variables $a_{1,1},\dots,a_{2k,2}\in \cA$. If $A\ge p^{1/2}$ then we have 
\begin{align*}
&D_k^{\times}(\cA)=\frac{A^{4k}}{p}\\
& \quad +O_k\left(\left(A^{4k-2}+p^{1-2^{-(k-1)}}A^{4k-4}E_{+}(\cA)^{2^{-(k-1)}}\right)\log^4{A}\right).
\end{align*}
\end{cor}
Using the trivial bound $E_{+}(\cA)\le A^3$ in Corollary~\ref{cor:Dk1} gives the following sharp asymptotic formula for $D_k^{\times}(\cA)$ for sets of cardinality a little larger than $p^{1/2}$.
\begin{cor}
\label{cor:Dksharp}
 For any $k\ge3$ and $A\ge p^{1/2+1/(2^{k+1}-6)}$  we have
\begin{align*}
D_k^{\times}(\cA)=\frac{A^{4k}}{p}+O_k\left(A^{4k-2}\log^4{A}\right) .
\end{align*}
\end{cor}

We define $N(\cX,\cY,\cZ)$ to be the number of solutions to
\begin{align*}
x_1(y_1-z_1)=x_2(y_2-z_2)
\end{align*}
with $x_1,x_2 \in \cX, y_1, y_2 \in \cY$ and $z_1,z_2 \in \cZ$. We now recall \cite[Corollary 2.4]{PetShp}.
\begin{lemma} \label{lem:NXYZ}
Let $\cX, \cY, \cZ \subset \F^*_p$ with $|\cX|=X, |\cY|=Y, |\cZ|=Z$ and $M=\max(X,Y,Z)$. Then
\begin{align*}
N(\cX,\cY,\cZ) \ll \frac{X^2Y^2Z^2}p + X^{3/2}Y^{3/2}Z^{3/2}+MXYZ.
\end{align*}
\end{lemma}

\subsection{Proof of Theorem \ref{thm:multlin2}}
\begin{proof}
Let 
\begin{align*}
S= S(\cX_1, \dots, \cX_n; \omega_1, \dots, \omega_n).
\end{align*}
By Lemma \ref{lem:SXin}, after permuting the variables, we have
\begin{align*}
|S|^{2^{n-1}} &\ll (X_1 \dots X_{n})^{2^{n-1}}\left(\frac{1}{X_{n-1}}+\dots+\frac{1}{X_{1}^{2^{n-2}}}\right) \\ &\qquad + (X_1\dots X_{n-1})^{2^{n-1}-2}X_n^{2^{n-1}-1}\sum_{\substack{x_1,y_1 \in\cX_1 \\ x_1 \neq y_1}}\dots \sum_{\substack{x_{n-1}, y_{n-1} \in\cX_{n-1}\\ x_{n-1} \neq y_{n-1}}} \\
&\qquad \qquad \qquad  \qquad \qquad \left|\sum_{x_n \in \cX_n}\ep(x_n(x_1-y_1)\dots (x_{n-1}-y_{n-1}))\right|.
\end{align*}
We now collect together $(x_2-y_2)\dots(x_{n-1}-y_{n-1})=\lambda$ and denote the number of solutions to this equation to be $J(\lambda)$. Similarly we collect $x_1(x_n-y_n) = \mu$ and we denote the number of solutions to this equation to be $I(\mu)$. Hence,
\begin{align*}
|S|^{2^{n-1}}&\ll_n (X_1 \dots X_{n})^{2^{n-1}}\left(\frac{1}{X_{n-1}}+\dots+\frac{1}{X_{1}^{2^{n-2}}}\right)\\
&\qquad + (X_1 \dots X_{n-1})^{2^{n-1}-2}X_n^{2^{n-1}-1} \sum_{\lambda \in \F^*_p} J(\lambda) \left|  \sum_{\mu \in \F_p } I(\mu) \ep(\lambda \mu)\right|\\
&= (X_1 \dots X_{n})^{2^{n-1}}\left(\frac{1}{X_{n-1}}+\dots+\frac{1}{X_{1}^{2^{n-2}}}\right)\\
&\qquad +(X_1 \dots X_{n-1})^{2^{n-1}-2}X_n^{2^{n-1}-1} \sum_{\lambda \in \F^*_p}   \sum_{\mu \in \F_p } J(\lambda) \eta_\lambda I(\mu) \ep(\lambda \mu)
\end{align*}
for some complex weight $\eta_\lambda$ with $|\eta_\lambda|=1$. Now, by Lemma \ref{lem:NXYZ} with $X=Y=X_1, Z=X_n$ we have
\begin{align*}
\sum_{\mu \in \F_p} I(\mu)^2 = N(\cX_n, \cX_1, \cX_1) \ll X_1^{3}X_n^{3/2}. 
\end{align*}
Similarly,
\begin{align*}
\sum_{\lambda \in \F_p^*} J(\lambda)^2 = D^{\times,*}_{n-2}(\cX_2, \dots, \cX_{n-1}).
\end{align*}
We apply Corollary \ref{cor:Dktimes} and \ref{cor:Dktimes2} combined with Lemma \ref{lem:Dktimes*} along with Lemma \ref{lem:bilin} to obtain
\begin{align*}
 |S|^{2^{n-1}}&\ll_n (X_1 \dots X_{n})^{2^{n-1}}\left(\frac{1}{X_{n-1}}+\dots+\frac{1}{X_{1}^{2^{n-2}}}\right) \\
& \qquad \qquad+ (X_1\dots X_n)^{2^{n-1}}p^{1/2} X_1^{-1/2}X_n^{-1/4}\left(\prod_{i=2}^{n-1}B_n(\cX_i)^{2^{n-1}}\right).
\end{align*}
This completes the proof.
\end{proof}

\subsection{Proof of Theorem \ref{thm:multlin3}}
We note that the conditions~\eqref{thm:multlin3cond} and Corollary~\ref{cor:Dksharp} imply that 
$$\widetilde D_{n}^{\times,*}(\cX_i)\ll (\log{p})^{4}X_i^{4n-2},$$
and hence by Lemma~\ref{lem:SMV1}

\begin{align*}
&|S(\cX_1, \dots, \cX_n; \omega_1, \dots, \omega_n)|^{2^{n}}\ll (X_1\dots X_n)^{2^{n}}\left(\frac{1}{X_1^{2^{n-1}}}+\dots+\frac{1}{X_n} \right) \\
& \qquad \qquad +(\log{p})^{4}p^{1/2}(X_1\dots X_n)^{2^n-1/n},
\end{align*}
from which the desired result follows.
\section{Multinomial Exponential Sums}
\subsection{Preliminaries}
The aim of this section is to extend the results of \cite{Mac2} and \cite{MSS}  beyond the cases of trinomials and quadrinomials, to more general multinomial sums.

We recall the following bound of \cite{MSS}.
\begin{lemma} \label{Dtimesgroup}
Let $\cG \subseteq\F_p^*$ be a multiplicative subgroup with $|\cG|=G$. Then
\begin{align*}
D_2^\times(\cG) - \frac{G^8}{p} \ll  \left\{
\begin{array}{ll}
p^{1/2} G^{\frac{11}{2}},& \text{if $ G \ge p^{\frac{2}{3}}$},\\
G^7 p^{-\frac{1}{2}} , & \text{if $p^{\frac{2}{3}} > G \ge p^{\frac{1}{2}}\log p$},\\
G^6 \log G, & \text{if $G< p^{\frac{1}{2}}\log p$}. 
\end{array}
\right.
\end{align*}
\end{lemma}
Combining with \eqref{eq:Dkk-1} and observing which term dominates we get the following corollary.
\begin{cor} \label{cor:Dtimesgroup}
Let $\cG \subseteq\F_p^*$ be a multiplicative subgroup with $|\cG|=G$. Then
\begin{align*}
D_k^\times(\cG)\ll  \left\{
\begin{array}{ll}
G^{4k}p^{-1} &\text{if $G \ge p^{\frac{1}{2}}\log p$},\\
G^{4k-2+o(1)}, & \text{if $G< p^{\frac{1}{2}}\log p$}. 
\end{array}
\right.
\end{align*}
\end{cor}
We also have the following result as a consequence of \cite[Lemma 2.4]{Mac2}.
\begin{lemma} \label{lem:NGHgroup}
Let $\cG, \cH \subset \F^*_p$ be multiplicative subgroups with cardinalities $G,H$ respectively with $G \ge H$. Then,
\begin{align*}
N(\cH, \cG, \cG) \ll  \left\{
\begin{array}{ll}
H^2G^{\frac{7}{2}}p^{-\frac{1}{2}} &\text{if $G \ge p^{\frac{1}{2}}\log p$},\\
H^2G^{\frac{5}{2}+o(1)}, & \text{if $G< p^{\frac{1}{2}}\log p$}. 
\end{array}
\right.
\end{align*}
\end{lemma}
We then have the following result on multilinear exponential sums over subgroups, which may be of independent interest to the reader.
\begin{lemma} \label{lem:multlingroup}
Let $\cX_i \subset \F_p$ be multiplicative subgroups with $|\cX_i|=X_i$, $X_1\ge X_2\ge\dots\ge  X_{n}$, $n\ge 4$. Then
\begin{align*}
S(\cX_1, \dots, \cX_n; \omega_1, \dots, \omega_n) &\ll_n (X_1\dots X_n)p^{\frac{1}{2^n}} A_n(\cX_1) \prod_{i=2}^{n-1} B_n(\cX_i) \\
&\qquad  + (X_1\dots X_n)\left(\frac{1}{X_n^{1/2}}+\dots+\frac{1}{X_1^{1/2^n}} \right)
\end{align*}
where
\begin{align*}
A_n(\cX_1) = \left\{
\begin{array}{ll}
X_1^{-\frac{1}{2^{n+1}}}p^{-\frac{1}{2^{n+1}}}, & \text{if $X_1 \ge p^{\frac{1}{2}}\log p$}, \\
X_1^{-\frac{3}{2^{n+1}}+o(1)}, & \text{if $X_1 < p^{\frac{1}{2}}\log p$,}
\end{array}
\right.
\end{align*}
and
\begin{align*}
B_n(\cX_i) = \left\{
\begin{array}{ll}
p^{-\frac{1}{2^{n}(n-2)}},& \text{if $ X_i \ge p^{\frac{1}{2}}\log p$},\\
X_i^{-\frac{1}{2^{n-1}(n-2)}},& \text{if $X_i< p^{\frac{1}{2}}\log p$.}
\end{array}
\right.
\end{align*}
\end{lemma}
\begin{proof}
The proof follows that of Theorem \ref{thm:multlin2}, however we use Corollary \ref{cor:Dtimesgroup} and Lemma \ref{lem:NGHgroup} in place of their relevant results on arbitrary sets.
\end{proof}
\subsection{Proof of Theorem \ref{thm:multinom}}
Let $\alpha_{k_i} = \gcd(k_i, p-1)$ for each $i=1, \dots, t$. We then let $\cG_{\alpha_i}$ be the subgroups of $\F^*_p$ generated by the elements of order $\alpha_{k_i}$. Then
\begin{align*}
T_\chi(\Psi) &= \frac{1}{\alpha_{k_1} \dots \alpha_{k_{t-1}}} \sum_{x_1 \in \cG_{\alpha_1}}\dots \sum_{x_{t-1} \in \cG_{\alpha_{t-1}}}\\
 & \qquad \qquad \qquad \qquad \sum_{x_t \in \F^*_p} \chi(x_1\dots x_t)\ep(\Psi(x_1 \dots x_t)) \\
 &= \frac{1}{\alpha_{k_1} \dots \alpha_{k_{t-1}}} \sum_{x_1 \in \cG_{\alpha_1}}\dots \sum_{x_{t-1} \in \cG_{\alpha_{t-1}}} \sum_{x_n \in \F^*_p} \chi(x_1\dots x_t)\\
 & \ep(a_1(x_2\dots x_t)^{k_1})\dots \ep(a_{t-1}(x_1\dots x_{t-2}x_t)^{k_{t-1}})\e_p(a_t(x_1\dots x_t)^{k_t})\\
 &= \frac{1}{\alpha_{k_1} \dots \alpha_{k_{t-1}}} \sum_{x_1 \in \cG_{\alpha_1}}\dots\sum_{x_{t-1} \in \cG_{\alpha_{t-1}}}\\
 & \qquad \sum_{x_t \in \F^*_p} \omega_1(\textbf{x}) \dots \omega_t (\textbf{x}) \ep(a_t(x_1\dots x_t)^{k_t}).
\end{align*}
Now the image $\cX_t = \{x_t^{k_t} : x_t \in \F^*_p\}$ of non-zero $k_tth$ powers contains $(p-1)/\alpha_{k_t}$ elements, each appearing with multiplicity $\alpha_{k_t}$. Similarly, the images $\cX_i = \{x_i^{k_t} : x_i \in \cG_{\alpha_{k_i}}\}$ contain $\alpha_{k_i}/\gcd(\alpha_{k_i}, \alpha_{k_t})$ elements, each appearing with multiplicity $\gcd(\alpha_{k_i}, \alpha_{k_t})$, for $i=1, \dots, t-1$. Hence, we apply Lemma \ref{lem:multlingroup} to obtain
\begin{align*}
&T_\chi(\Psi) \ll_t \frac{\alpha_{k_t}}{\beta_{k_1} \dots \beta_{k_{t-1}}} \cdot\left (p^{\frac{1}{2^t}}\beta_{k_{t-1}} A_t\left(\frac{p-1}{\alpha_{k_t}}\right) \prod_{i=1}^{t-2}B_t(\beta_{k_i}) \right) \\
& + \frac{\alpha_{k_t}}{\beta_{k_1} \dots \beta_{k_{t-1}}} \cdot \frac{p-1}{\alpha_{k_t}}\beta_{k_1} \dots\beta_{k_{t-1}}\left( \left(\frac{\alpha_{k_t}}{p-1}\right)^{\frac{1}{2}}+\beta_{k_1}^{\frac{-1}{2^2}} +\dots+\beta_{k_{t-1}}^{\frac{-1}{2^{t}}}\right).
\end{align*}
By simplifying we reach the required result.

\section{Weyl Sums Over Generalized Arithmetic Progressions}
\subsection{Preliminaries}
We will require an estimate for the $\ell_1$ norm of the Fourier transform of proper generalized arithmetic progressions which is due to Shao~\cite{Shao1}.
\begin{lemma}
\label{lem:genAPf}
Let $\cA\subseteq \F_p$ be a proper generalized arithmetic progression of rank $r$ and let $\widehat \cA(z)$ denote the Fourier transform of $\cA$. Then we have 
\begin{align*}
\sum_{z=1}^{p}|\widehat \cA(z)|\ll_r p(\log{p})^r.
\end{align*} 
\end{lemma}
The following is due to Bourgain~\cite[Theorem~A]{Bour2}.
\begin{lemma}
\label{lem:bourml}
Let $0<\delta<1/4$ and $r\ge 2$. There exists some $\delta'$ such that if $p$ is a sufficiently large prime and $A_1,\dots,A_r\subseteq \F_p$ satisfy
\begin{align*}
& |A_i|>p^{\delta} \\
&\prod_{i=1}^{r}|A_i|>p^{1+\delta},
\end{align*}
then
\begin{align*}
\left|\sum_{x_1\in A_1,\dots x_r\in A_r}e_p(x_1\dots x_r) \right|\ll |A_1|\dots |A_r|p^{-\delta'}.
\end{align*}
\end{lemma}
\subsection{Proof of Theorem~\ref{thm:genap}}
Considering the sum
\begin{align*}
S=\sum_{a\in \cA}e_p(F(a)),
\end{align*}
we note that for any $a_2,\dots,a_d\in \cA$
\begin{align}
\label{eq:Sgap}
S&=\frac{1}{p}\sum_{z=1}^{p}\left(\sum_{a\in \cA}e_p(-za)\right) \\ & \times\sum_{a_1\in d\cA}e_p(F(a_1-a_2-\dots-a_d)+z(a_1-a_2-\dots-a_d)),
\end{align}
where $d\cA$ denotes the $d$-fold sumset 
\begin{align*}
d\cA=\{ a_1+\dots+a_d \ : a_i\in \cA \},
\end{align*}
so that
\begin{align}
\label{eq:dAbound}
|d\cA|\ll |\cA|.
\end{align}
Averaging~\eqref{eq:Sgap} over $a_2,\dots,a_d\in \cA$ and using Lemma~\ref{lem:genAPf} gives
\begin{align}
\label{eq:SthmGenAP}
S\ll \frac{1}{p}\sum_{z=1}^{p}|\widehat \cA(z)|\frac{|T(z)|}{|\cA|^{d-1}}\ll_r \frac{p^{o(1)}|T(z_0)|}{|\cA|^{d-1}},
\end{align}
for some $z_0\in \F_p$, where 
\begin{align*}
T(z)=\sum_{\substack{a_1\in d\cA \\ a_2,\dots,a_d\in \cA}}e_p(F(a_1-a_2-\dots-a_d)+z(a_1-a_2-\dots-a_d)).
\end{align*}
Since $F$ has degree $d$, we may write
\begin{align*}
F(z)=\sum_{i=0}^{d}b_iz^{i},
\end{align*}
and hence 
\begin{align*}
& F(a_1-a_2-\dots-a_d)=\sum_{i=0}^{d}b_i(a_1-a_2-\dots-a_d)^{i} \\ &  \quad \quad \quad =F_1(a_2,\dots,a_d)+\dots+F_d(a_1,\dots,a_{d-1})+(-1)^{d-1}b_da_1\dots a_d,
\end{align*}
for some sequence of polynomials $F_1,\dots,F_d$ where $F_i$ is independent of the variable $a_i$. This implies that 
\begin{align*}
&T(z_0)= \\ & \sum_{\substack{a_1\in d\cA \\ a_2,\dots,a_d\in \cA}}\omega_1(a_2,\dots,a_d)\dots \omega_d(a_1,\dots,a_{d-1})e_p((-1)^{d-1}b_da_1\dots a_d),
\end{align*}
for some sequence of weights $\omega_1,\dots,\omega_d$ with $\omega_i$ independent of $a_i$. By Lemma~\ref{lem:SXin}
\begin{align*}
|T(z_0)|^{2^{d-1}}& \ll_d |\cA|^{d2^{d-1}-1} \\ &+|\cA|^{d2^{d-1}-2d+1}\sum_{\substack{a_j,a_j'\in \cA \\ j\ge 2}}\left|\sum_{a_1\in d\cA}e_p(b_da_1(a_2-a_2')\dots (a_d-a_d')) \right|,
\end{align*}
and by the Cauchy-Schwarz inequality 
\begin{align*}
&|T(z_0)|^{2^{d}}\ll_d |\cA|^{d2^{d}-2} \\ & +|\cA|^{d2^{d}-2d}\left|\sum_{\substack{a_j,a_j'\in \cA \\ j\ge 2}}\sum_{a_1,a_1'\in d\cA}e_p(b_d(a_1-a_1')(a_2-a_2')\dots (a_d-a_d')) \right|.
\end{align*}
This implies that 
\begin{align*}
|T(z_0)|^{2d}\ll_d  |\cA|^{d2^{d}-2}+|\cA|^{d2^{d}-d}\left|\sum_{\substack{a_j\in \cA+c_j \\ j\ge 2}}\sum_{a_1\in d\cA+c_1}e_p(b_da_1\dots a_d) \right|,
\end{align*}
for some $c_1,\dots,c_d\in \F_p$. We note that the assumption~\eqref{eq:genAPcond} implies that the conditions of Lemma~\ref{lem:bourml} are satisfied and hence 
\begin{align*}
|T(z_0)|^{2d}\ll_d |\cA|^{d2^{d}}p^{-\delta'},
\end{align*}
for some $\delta'>0$ depending on $\varepsilon$ and the result follows from~\eqref{eq:SthmGenAP}.
%\bibliographystyle{amsplain}
%\bibliography{References}

\end{document}